\pdfoutput=1
\documentclass[12pt, reqno]{amsart}

\numberwithin{equation}{section}

\usepackage{amsmath,amsfonts,amssymb}
\usepackage{mathtools}
\usepackage[utf8]{inputenc}
\usepackage[english]{babel}
\usepackage{hyperref}
\usepackage{a4wide}
\usepackage{cite}
\usepackage{stmaryrd}
\usepackage[new]{old-arrows}

\DeclareUnicodeCharacter{FB01}{fi}
\newtheorem{theorem}{Theorem}[section]

\newtheorem{remark}[theorem]{Remark}
\newtheorem{example}[theorem]{Example}
\newtheorem{prop}[theorem]{Proposition}
\newtheorem*{aim-non}{Aim}

\newtheorem*{conjecture-non}{Conjecture}

\theoremstyle{definition}
\newtheorem{definition}[theorem]{Definition}

\begin{document}

	\title[Topological equivalences for $G$-bundles]{On some topological equivalences for moduli spaces of $G$-bundles}

	\author[S. Roy]{Sumit Roy}
\address{Stat-Math Unit, Indian Statistical Institute, 203 B.T. Road, Kolkata 700 108, India.}
 \email{sumitroy\_r@isical.ac.in}
\thanks{E-mail: sumit.roy061@gmail.com}
\thanks{Address: Stat-Math Unit, Indian Statistical Institute, 203 B.T. Road, Kolkata 700 108, India.}

 \begin{abstract}
     Let $X$ be a smooth projective curve of genus $g \geq 3$, and let $G$ be a nontrivial connected reductive affine algebraic group over $\mathbb{C}$. Examining the moduli spaces of regularly stable $G$-bundles and holomorphic $G$-connections with a fixed topological type $d\in \pi_1(G)$ over $X$, we establish that the $k$-th homotopy groups of these two moduli spaces are isomorphic for $k \leq 2g-4$. Lastly, we explicitly describe the homotopy groups of the moduli space of $\mathrm{SL}(n,\mathbb{C})$-connections over $X$. We also prove that the mixed Hodge structure on the cohomology groups (modulo torsion) with integer coefficients on the moduli spaces of $\mathrm{SL}(n,\mathbb{C})$-bundles and of $\mathrm{SL}(n,\mathbb{C})$-connections are pure and isomorphic.
 \end{abstract}

\maketitle

\section{Introduction}
Let $X$ be a smooth projective curve (or a compact Riemann surface) of genus $g \geq 3$, and let $G$ be a nontrivial connected reductive affine algebraic group over $\mathbb{C}$. In this article, we are interested in the moduli spaces of semistable $G$-Higgs bundles and holomorphic $G$-connections of a fixed topological type $d \in \pi_1(G)$ over $X$, denoted by $\mathcal{M}^d_{\mathrm{Higgs}}(G)$ and $\mathcal{M}^d_{\mathrm{conn}}(G)$, respectively. Notably, these moduli spaces are not smooth in general. However, the regularly stable locus of $\mathcal{M}^d_{\mathrm{Higgs}}(G)$, where the automorphism group of the $G$-Higgs bundle coincides with the center of $G$, is non-singular (same is true for $\mathcal{M}^d_{\mathrm{conn}}(G)$).

Simpson, in \cite{Si94a}, introduced a family termed the Hodge moduli space over $\mathbb{C}$, with fibers over $0$ and $1$ precisely corresponding to the moduli of Higgs bundles and holomorphic connections, respectively. He established a homeomorphism between the moduli space of Higgs bundles and holomorphic connections, known as the non-abelian Hodge correspondence (refer to \cite{S92}, \cite{Si94a}, \cite{S94}). While these moduli spaces generally possess singularities, when considering cases where the rank and degree are coprime, they exhibit smoothness. In \cite[Theorem 6.2]{HT03}, Hausel and Thaddeus specifically examined the coprime rank and degree scenario, establishing the pure Hodge structure of these moduli spaces. 

 A principal $G$-bundle $E$ is called \textit{regularly stable} if it is stable and $\mathrm{Aut}(E) = Z(G)$, where $Z(G)$ is the center of $G$. In \cite{BH12}, Biswas and Hoffmann showed that the moduli space $\mathcal{M}^{d,rs}(G)$ of regularly stable $G$-bundles is the smooth locus of $\mathcal{M}^{d}(G)$. Similalrly, we can define the regularly stable $G$-connection. Consider the subset
    \[
\mathcal{U}^{d,rs}_\mathrm{conn}(G) \subset \mathcal{M}^{d,rs}_\mathrm{conn}(G)
    \]
  which consists of regularly stable $G$-connections whose underlying $G$-bundle is regularly stable. Then consider the forgetful map
  \begin{equation*}
       p: \mathcal{U}^{d,rs}_\mathrm{conn}(G) \longrightarrow \mathcal{M}^{d,rs}(G)
  \end{equation*}
which forgets the connection structure on the $G$-bundle. Using this map $p$, we establish some topological invariances. For instance, in Theorem \ref{maintheorem}, we show that
\[
   \pi_k(\mathcal{M}^{d,rs}_\mathrm{conn}(G)) \cong \pi_k(\mathcal{M}^{d,rs}(G))
\]
for all $k=0,\dots, 2g-4$. Then we consider the moduli space $\mathcal{M}^d_\mathrm{Hod}(G)$, which consists of triples $(E_G,\lambda,\nabla)$, where $\lambda \in \mathbb{C}$, $E$ is a $G$-bundle of topological type $d\in \pi_1(G)$ over $X$, and $\nabla$ is a semistable $\lambda$-connection on $E$ (see \cite{BGH13} for more details) and let
\[
\mathcal{M}^{d,rs}_\mathrm{Hod}(G) \subset \mathcal{M}^d_\mathrm{Hod}(G)
\]
be the regularly stable locus of triples $(E,\lambda,\nabla)$.

In section \ref{lastsection}, we restrict our attention to $G=\mathrm{GL}_n$ and $\mathrm{SL}_n$. In this case, the notion of regularly stable and stable coincides. Therefore, if the rank $n$ and degree $d$ are coprime, then the moduli spaces of semistable $G$-Higgs bundles and of semistable $G$-connections are smooth. In Theorem \ref{homotopy}, we explicitly compute the homotopy groups of the moduli space of $\mathrm{SL}(n,\mathbb{C})$-connections.  Using the smoothness of 
\begin{align*}
\begin{split}
\pi: \mathcal{M}^{d,rs}_\mathrm{Hod}(G) &\longrightarrow \mathbb{C}\\
(E,\lambda,\nabla) &\longmapsto \lambda,
\end{split}
\end{align*}
we prove that the mixed Hodge structures on the cohomology groups (with rational coefficients) of $\mathcal{M}^{d,rs}_\mathrm{Higgs}(G)$ and $\mathcal{M}^{d,rs}_\mathrm{conn}(G)$ are pure and isomorphic for $G=\mathrm{SL}(n,\mathbb{C})$. In Theorem \ref{integer}, we prove that the mixed Hodge structure on the cohomology groups (modulo torsion) with integer coefficients on the moduli spaces of $\mathrm{SL}(n,\mathbb{C})$-bundles and of $\mathrm{SL}(n,\mathbb{C})$-connections are pure and isomorphic.

\section{Preliminaries}
\subsection{Principal $G$-bundles}
Let $G$ be a connected reductive complex affine algebraic group with Lie algebra $\mathfrak{g}=\mathrm{Lie}(G)$. Let $X$ be a smooth projective curve of genus $g\geq 3$.
\begin{definition}
     A \textit{holomorphic principal $G$-bundle} $\pi: E \to X$ with structure group $G$ (shortly, a $G$-bundle) is a complex manifold $E$, equipped with a holomorphic right $G$-action that is free, such that the holomorphic projection map $\pi$ is $G$-equivariant ($X$ is equipped with the trivial action).
\end{definition}
  It is well known that $H^1(X,\underline{G})$ parameterizes the set of isomorphism classes of principal $G$-bundles on $X$, where $\underline{G}$ is the natural sheaf induced by $G$ on $X$. This leads to the well-known case of vector bundles when $G = GL(n,\mathbb{C})$. The adjoint action
  \begin{align*}
	    	\mathrm{ad} : G \longrightarrow \operatorname{End}(\mathfrak{g}).
\end{align*}
  of $G$ on $\mathfrak{g}$, together with the right $G$-action on $E$ gives a $G$-action on $E\times \mathfrak{g}$, defined by
  \[
	(v,\xi)\cdot g = (v\cdot g, \mathrm{ad}(g^{-1})(\xi)), \hspace{0.2cm} \forall \hspace{0.1cm} (v,\xi)\in E\times \mathfrak{g},\hspace{0.1cm} g\in G.
	\] 
We denote the associated \textit{adjoint bundle} over $X$ by
\[
\mathrm{ad}(E) \coloneqq E \times^G \mathfrak{g} \coloneqq (E \times \mathfrak{g})/G.
\]

\begin{definition}
    A principal $G$-bundle $E$ is said to be \textit{stable} (or \textit{semistable}) if for each maximal parabolic subgroup $P \subset G$ and any holomorphic reduction $E_P$ that reduces the structure group of $E$ to $P$,
    \[
	\deg(\mathrm{ad}(E_P)) < 0 \hspace{0.2cm} (\mathrm{respectively,} \hspace{0.2cm} \leq 0 \hspace{0.05cm})
	\]
 where $\mathrm{ad}(E_P) \subset \mathrm{ad}(E)$ denotes the adjoint bundle of $E_P$.

 A principal $G$-bundle $E$ is called \textit{regularly stable} if it is stable and the natural map
 \[
 Z(G) \longrightarrow \mathrm{Aut}(E)
 \]
 given by the action of $G$ on $E$ is indeed an isomorphism, i.e. 
 \[
 \mathrm{Aut}(E) \cong Z(G).
 \]
\end{definition}

For a principal $G$-bundle $E$, its topological type corresponds to an element of the finite abelian group $\pi_1(G)$ (see \cite{R75}). The moduli space $\mathcal{M}^d(G)$ of semistable holomorphic principal $G$-bundles over $X$ of topological type $d\in \pi_1(G)$ is an irreducible normal projective variety over $\mathbb{C}$ of dimension
\[
\dim \mathcal{M}^d(G) = (g-1)\cdot \dim_{\mathbb{C}}G + \dim_{\mathbb{C}} Z(G)
\]
(see \cite{R75, R96}). This moduli space has singularities in general (even the stable locus has singularities). But the regularly stable locus $$\mathcal{M}^{d, rs}(G) \subset \mathcal{M}^d(G)$$ is the smooth open locus of the moduli space $\mathcal{M}^d(G)$. Therefore, we will restrict our attention to the regularly stable locus $\mathcal{M}^{d, rs}(G)$ (see \cite[Corollary $3.4$]{BH12}.

\begin{remark}\label{remark}
For $G=\mathrm{GL}(n,\mathbb{C})$, the topological type $d\in \pi_1(G)$ corresponds to the degree $d$ of the vector bundle (or, principal $\mathrm{GL}(n,\mathbb{C})$-bundle). Also, in this case, the stable locus is precisely the smooth locus of $\mathcal{M}^d(\mathrm{GL(n,\mathbb{C})})$. This is also true for the moduli space $\mathcal{M}^d(\mathrm{SL(n,\mathbb{C})})$ of $\mathrm{SL(n,\mathbb{C})}$-bundles.    
\end{remark} 

\subsection{Principal $G$-Higgs bundles}
Let $K_X$ denote the cotangent bundle of the Riemann surface $X$.
\begin{definition}
A principal \textit{$G$-Higgs bundle} over a Riemann surface $X$ is defined as a pair $(E,\varphi)$, where $E$ represents a holomorphic principal $G$-bundle and 
    \[
    \varphi \in H^0(X, \mathrm{ad}(E_G) \otimes K_X)
    \]
    is a holomorphic section, referred to as the \textit{Higgs field} \cite{H87, S92}. 
\end{definition}
\begin{example}
    A principal $\mathrm{GL}(n,\mathbb{C})$-Higgs bundle over $X$ is a pair $(E,\varphi)$, which consists of a holomorphic vector bundle $E$ of rank $n$ and a holomorphic section $\varphi \in H^0(X, \mathrm{End}(E) \otimes K_X)$. This fundamental concept of the Higgs bundle was initially introduced by Hitchin \cite {H87a}.
    
    Similarly, a  $\mathrm{SL}(n,\mathbb{C})$-Higgs bundle over $X$ is a pair $(E,\varphi)$, which consists of a holomorphic vector bundle $E$ of rank $n$ with $\det (E) = \mathcal{O}$ and $\varphi \in H^0(X, \mathrm{End}(E) \otimes K_X)$ a holomorphic section with $\mathrm{trace}(\varphi)=0$.
\end{example}

\begin{definition}
    A principal \textit{$G$-Higgs bundle} $(E,\varphi)$ is said to be \textit{semistable} (respectively, \textit{stable}) if for every holomorphic reduction $E_P$ of the structure group of $E$ to a $\varphi$-invariant maximal parabolic subgroup $P\subsetneq G$, i.e.  $\varphi \in H^0(X, \mathrm{ad}(E_P) \otimes K_X)$ we have
\[
	\deg(\mathrm{ad}(E_P)) \leq 0 \hspace{0.2cm} (\mathrm{respectively,} \hspace{0.2cm} < 0 \hspace{0.05cm}).
\]
A principal $G$-Higgs bundle $(E,\varphi)$ over $X$ is said to be \textit{regularly stable} if it is stable and the automorphism group of $(E,\varphi)$ coincides with the center of $G$.
\end{definition}

The moduli space $\mathcal{M}^d_{\mathrm{Higgs}}(G)$ of semistable $G$-Higgs bundles of fixed topological type $d\in \pi_1(G)$ is an irreducible normal quasi-projective variety of dimension
\[
\dim \mathcal{M}^d_{\mathrm{Higgs}}(G) = 2(g-1)\cdot \dim_{\mathbb{C}}G + 2 \dim_{\mathbb{C}}Z(G)
\]
(see \cite{S94} for more details). Moreover, the regularly stable locus $\mathcal{M}^{d,rs}_{\mathrm{Higgs}}(G) \subset \mathcal{M}^d_{\mathrm{Higgs}}(G)$ is open and smooth.  The tangent space of $\mathcal{M}^{d,rs}(G)$ at the point $E$ can be identified as isomorphic to $H^1(X, \mathrm{ad}(E))$ using deformation theory. Employing Serre duality, we can establish that
\[
H^0(X, \mathrm{ad}(E) \otimes K_X) \cong H^1(X, \mathrm{ad}(E))^*.
\]
Consequently, the cotangent bundle of $\mathcal{M}^{d,rs}_{\mathrm{Higgs}}(G)$, denoted as $$T^*\mathcal{M}^{d,rs}(G) \subset \mathcal{M}^{d,rs}_{\mathrm{Higgs}}(G),$$ forms an open and dense subvariety within $\mathcal{M}^{d,rs}_{\mathrm{Higgs}}(G)$. As a result, we obtain
\[
\dim \mathcal{M}^{d,rs}_{\mathrm{Higgs}}(G) = 2\dim \mathcal{M}^{d,rs}(G).
\]

This moduli space $\mathcal{M}^d_{\mathrm{Higgs}}(G)$ is equipped with a natural $\mathbb{C}^*$-action, given by
\begin{equation}\label{actionHiggs}
    t\cdot (E,\varphi) \coloneqq (E, t\varphi).
\end{equation}

\subsection{Holomorphic $G$-conections}
Let $q: E \to X$ represent the projection morphism from the total space of a $G$-bundle $E$ to $X$. For any open subset $U \subset X$, consider the space $\mathcal{F}(U)$ consisting of $G$-equivariant vector fields on $q^{-1}(U)$. Let $\mathcal{F}$ be the coherent sheaf on $X$ associating each $U$ with the vector space $\mathcal{F}(U)$. The vector bundle corresponding to the sheaf $\mathcal{F}$ is called the \textit{Atiyah bundle} for the $G$-bundle $E$ and is denoted as $\mathrm{At}(E)$ (refer to \cite{A57}). Essentially, it is obtained as the quotient
\[
 \mathrm{At}(E) \coloneqq (TE)/G
\]
where $TE$ denotes the (holomorphic) tangent bundle of $E$; thus, the Atiyah bundle $\mathrm{At}(E)$ forms a (holomorphic) vector bundle over $E/G = X$. Consequently, there exists an exact sequence of vector bundles
\begin{equation}\label{Atiyah}
0 \longrightarrow \mathrm{ad}(E) \longrightarrow \mathrm{At}(E) \overset{\tau}\longrightarrow TX \longrightarrow 0,
\end{equation}
where $TX$ represents the tangent bundle of $X$. The surjective morphism $\tau: \mathrm{At}(E) \to TX$ is constructed using the differential $dq$ of $$q: E \longrightarrow X.$$ It's important to note that the adjoint bundle $\mathrm{ad}(E)$ is the holomorphic subbundle of the tangent bundle $TE$ defined by the kernel of the differential $dq$. The short exact sequence (\ref{Atiyah}) above is called the \textit{Atiyah exact sequence} for the $G$-bundle $E$.

A \textit{holomorphic connection} on a $G$-bundle $E$ can be described as a holomorphic splitting of the above Atiyah exact sequence (\ref{Atiyah}), i.e., a holomorphic morphism 
\[ 
\mathcal{\nabla} : TX \longrightarrow \mathrm{At}(E) 
\]
satisfying $$\tau \circ \mathcal{\nabla} = \mathrm{id}_{TX}$$ within the context of the Atiyah exact sequence (\ref{Atiyah}). If $\mathcal{\nabla}'$ stands as another holomorphic splitting of (\ref{Atiyah}), then $$\mathcal{\nabla}-\mathcal{\nabla}': TX \longrightarrow \mathrm{ad}(E)$$ is a holomorphic holomorphism. Conversely, for any holomorphic section $$s \in H^0(X, K_X \otimes \mathrm{ad}(E)),$$ if $\mathcal{\nabla}$ is a holomorphic splitting of (\ref{Atiyah}), then so is $\mathcal{D}+s$. Consequently, the space of all holomorphic connections on the principal $G$-bundle $E$ is an affine space for $H^0(X, K_X \otimes \mathrm{ad}(E))$.

Since $\dim_\mathbb{C}(X)=1$, any holomorphic connection on $E$ inherently becomes a flat connection on $E$ that is compatible with its holomorphic structure. Since the space of all extensions of $TX$ by $\mathrm{ad}(E)$ is parametrized by $H^1(X, K_X \otimes \mathrm{ad}(E))$, the requisite condition for the existence of a flat connection on the principal $G$-bundle $E$ is equivalent to asserting that if $\beta \in H^1(X, K_X \otimes \mathrm{ad}(E))$ corresponds to the Atiyah sequence (\ref{Atiyah}), then $\beta=0$ (see \cite{AB02} for more details). Following \cite[Theorem $4.1$]{AB02}, we get that a holomorphic connection on the $G$-bundle $E$ always exists if $E$ is semistable. 

\begin{definition} 
A \textit{holomorphic $G$-connection} is defined as a pair $(E,\mathcal{\nabla})$ where $E$ is a principal $G$-bundle, and $\mathcal{\nabla}$ is a holomorphic connection on $E$. 
\end{definition}

Since a vector bundle with a flat connection has degree zero, a holomorphic $G$-connection $(E, \mathcal{\nabla})$ is always semistable. Let $\mathcal{M}^d_\mathrm{conn}(G)$ be the moduli space of holomorphic $G$-connections with a fixed topological type $d\in \pi_1(G)$ over $X$. It is a normal irreducible quasi-projective complex variety of dimension 
\[
\dim \mathcal{M}^d_{\mathrm{conn}}(G)= 2(g-1)\cdot \dim_{\mathbb{C}} G + 2 \dim_{\mathbb{C}}Z(G)
\]
(see \cite{BGH13} for details). By \cite{BGH13}, the regularly stable locus 
\[
\mathcal{M}^{d,rs}_{\mathrm{conn}}(G) \subset \mathcal{M}^d_{\mathrm{conn}}(G)
\]
is open and smooth.

\begin{definition} A $\textit{holomorphic } \lambda$\textit{-connection} ($\lambda \in \mathbb{C}$) on a $G$-bundle $E$ over $X$ is defined as a holomorphic morphism of vector bundles
\[
\nabla : TX \longrightarrow \mathrm{At}(E)
\]
such that $$\tau \circ \nabla = \lambda\cdot \mathrm{id}_{TX},$$ where $\tau$ is as in the Atiyah sequence (\ref{Atiyah}).
\end{definition}

It can be checked that for $\lambda \neq 0$, a holomorphic $G$-connection $(E,\nabla)$ is always semistable (as if $\nabla$ is a $\lambda$-connection on $E$, then $\lambda^{-1}\nabla$ is also a holomorphic $G$-connection on $E$). Consider the moduli space $\mathcal{M}^d_\mathrm{Hod}(G)$, which consists of triples $(E_G,\lambda,\nabla)$, where $\lambda \in \mathbb{C}$, $E$ is a $G$-bundle of topological type $d\in \pi_1(G)$ over $X$, and $\nabla$ is a semistable $\lambda$-connection on $E$ (refer to \cite{BGH13},\cite{S94} for details). Let
\[
\mathcal{M}^{d,rs}_\mathrm{Hod}(G) \subset \mathcal{M}^d_\mathrm{Hod}(G)
\]
be the regularly stable locus of triples $(E,\lambda,\nabla)$, and as before, it is an open and smooth subvariety of $\mathcal{M}^d_\mathrm{Hod}(G)$. There exists a canonical surjective algebraic morphism
\begin{align}\label{proj}
\begin{split}
\pi: \mathcal{M}^d_\mathrm{Hod}(G) &\longrightarrow \mathbb{C}\\
(E,\lambda,\nabla) &\longmapsto \lambda,
\end{split}
\end{align}
whose fiber $\pi^{-1}(0)$ over the origin $0 \in \mathbb{C}$ corresponds to the moduli space of semistable principal $G$-Higgs bundles over $X$, i.e.
\[
\mathcal{M}^d_{\mathrm{Higgs}}(G) = \pi^{-1}(0) \subset \mathcal{M}^d_\mathrm{Hod}(G). 
\]
When $\lambda = 1$, we get
\[
\mathcal{M}^d_{\mathrm{conn}}(G) = \pi^{-1}(1)  \subset \mathcal{M}^d_\mathrm{Hod}(G).
\]
If we restrict $\pi$ to the regularly stable locus
\begin{equation}\label{rsproj}
    \pi: \mathcal{M}^{d, rs}_\mathrm{Hod}(G) \longrightarrow \mathbb{C}
\end{equation}
(we denote it by $\pi$ again for notational convenience), we get
\[
\pi^{-1}(0) \cong \mathcal{M}^{d, rs}_{\mathrm{Higgs}}(G)
\]
and 
\[
\pi^{-1}(1) \cong \mathcal{M}^{d, rs}_{\mathrm{conn}}(G).
\]
The natural $\mathbb{C}^*$-action (\ref{actionHiggs}) on the moduli space $\mathcal{M}^d_{\mathrm{Higgs}}(G)$ can be extended to a $\mathbb{C}^*$-action on the Hodge moduli space $\mathcal{M}^d_\mathrm{Hod}(G)$ defined by
\begin{equation}\label{action}
t\cdot (E,\lambda,\nabla) \coloneqq (E,t\lambda,t\nabla).
\end{equation}
Then, it is straightforward to see that the surjective submersion $\pi$ given in (\ref{proj}) is $\mathbb{C}^*$-invariant for the natural $\mathbb{C}^*$-action on $\mathbb{C}$.

\subsection{Mixed Hodge structures}
\begin{definition}
    A \textit{mixed Hodge structure} on a finitely generated free $\mathbb{Z}$-module $H_\mathbb{Z}$ is endowed with:
    \begin{enumerate}
        \item an increasing filtration (weight filtration)
        \[
        0=W_0 \subset W_1 \subset \cdots \subset W_k=H_\mathbb{Q} \coloneqq H_\mathbb{Z} \otimes \mathbb{Q}
        \]
        \item a decreasing filtration (the Hodge filtration)
        \[
           H \coloneqq H_\mathbb{Z}\otimes \mathbb{C} = F^0 \supset F^1 \supset F^2 \cdots \supset F^l=0
        \]
        induced by $F^\bullet$ on the complexified graded components of $W_\bullet$ providing each graded component $\mathrm{Gr}^pW_\bullet \coloneqq W_p/W_{p-1}$ with a pure Hodge structure of weight $p$, i.e. , for all $0\leq k \leq p$ we have
        \[
        \mathrm{Gr}^pW_\bullet^\mathbb{C} = F^k\mathrm{Gr}^pW_\bullet^\mathbb{C} \oplus \overline{F^{p-k+1}\mathrm{Gr}^pW_\bullet^\mathbb{C}}.
        \]
    \end{enumerate}
\end{definition}
Deligne in \cite{D72, D75} showed that the cohomology of a complex algebraic variety is equipped with a mixed Hodge structure. He proved the following:
\begin{prop}[\cite{D72}, \cite{D75}]
    Let $Y$ be a complex algebraic variety. Then a mixed Hodge structure exists on the cohomology $H^i(Y,\mathbb{C})$. Furthermore, the weight filtration is given by
    \[
    0=W_{-1} \subset W_0 \subset \cdots \subset_{2i} = H^(Y,\mathbb{Q})
    \]
    and the Hodge filtration is given by
    \[
    H^i(Y,\mathbb{C}) =F^0 \supset F^1 \supset \cdots \supset F^{i}\supset F^{i+1} = 0.
    \]
   \end{prop} 
    The compactly supported cohomology $H^*_c(Y,\mathbb{Q})$ also has a mixed Hodge structure. Also, there is a Poincar\'e duality in the context of mixed Hodge structure
  \begin{prop}[Poincar\'e duality]\label{poincare} Let $Y$ be a smooth connected complex algebraic variety of dimension $d$. Then the Poincar\'e duality is established as follows:
    \[
    H^i(Y,\mathbb{Q}) \times H^{2d-i}_c(Y,\mathbb{Q}) \longrightarrow H^{2d}_c(Y,\mathbb{Q}) \cong \mathbb{Q}(-d)
    \]
   and this duality is consistent with mixed Hodge structures. Here $\mathbb{Q}(-d)$ denotes the pure mixed structure on $\mathbb{Q}$ of weight $2d$ and a Hodge filtration such that $F^d=\mathbb{Q}$ and $F^{d+1}=0$.
\end{prop}
\begin{proof}
    See \cite{PS08} for a proof.
\end{proof}

\section{Topological Invariances}

Recall that $X$ is a smooth projective curve of $g\geq 3$, $G$ a connected reductive affine algebraic group over $\mathbb{C}$ and $\mathcal{M}^d(G)$ denote the moduli space of semistable $G$-bundles with regularly stable locus $\mathcal{M}^{d, rs}(G)$.
\begin{theorem}
    The natural map
    \[
    \pi_1(\mathcal{M}^{d,rs}(G)) \longrightarrow \pi_1(\mathcal{M}^d(G))
    \]
    which is induced by the inclusion $\mathcal{M}^{d,rs}(G) \xhookrightarrow{} \mathcal{M}^d(G)$ is surjective.
\end{theorem}
\begin{proof}
    We know that the regularly stable locus
    \[
    \mathcal{M}^{d,rs}(G) \subset \mathcal{M}^d(G)
    \]
    is the smooth locus of $\mathcal{M}^d(G)$ and the moduli space $\mathcal{M}^d(G)$ is a normal projective variety. We want to show that the homomorphism
    \[
    \pi_1(\mathcal{M}^{d,rs}(G)) \longrightarrow \pi_1(\mathcal{M}^d(G))
    \]
    induced from the inclusion $\mathcal{M}^{d,rs}(G) \xhookrightarrow{} \mathcal{M}^d(G)$ is surjective. Now consider a desingularization
    \[
    f: \widehat{\mathcal{M}^d(G)} \longrightarrow \mathcal{M}^d(G).
    \]
    Then Zariski's main theorem tells us that all fibers of $f$ are connected (see \cite[p. $280$, Corollary $11.4$]{H77}). As a consequence, the induced homomorphism
    \[
    f_*: \pi_1(\widehat{\mathcal{M}^d(G)}) \longrightarrow \pi_1(\mathcal{M}^d(G))
    \]
    is surjective. Moreover, since $\mathcal{M}^{d,rs}(G)$ is smooth, the induced map
    \[
    \pi_1(f^{-1}(\mathcal{M}^{d,rs}(G)) \longrightarrow \pi_1(\widehat{\mathcal{M}^d(G)})
    \]
    is surjective. Since the restriction map
    \[
    \left.f\right|_{f^{-1}(\mathcal{M}^{d,rs}(G))}: f^{-1}(\mathcal{M}^{d,rs}(G)) \longrightarrow \mathcal{M}^{d,rs}(G)
    \]
    is an isomorphism, by combining these morphisms we get that
    \[
    \pi_1(\mathcal{M}^{d,rs}(G)) \longrightarrow \pi_1(\mathcal{M}^d(G))
    \]
    is surjective.
\end{proof}

\begin{theorem}\label{maintheorem}
   Let $X$ be a smooth projective curve of $g\geq 3$ and $G$ a connected reductive affine algebraic group over $\mathbb{C}$. Let $\mathcal{M}^{d,rs}(G)$ (resp. $\mathcal{M}^{d,rs}_\mathrm{conn}(G)$) denote the moduli spaces of regularly stable $G$-bundles (resp. $G$-connections) of topological type $d\in \pi_1(G)$ over $X$. Then
   \[
   \pi_k(\mathcal{M}^{d,rs}_\mathrm{conn}(G)) \cong \pi_k(\mathcal{M}^{d,rs}(G))
   \]
   for all $k=0,\dots, 2g-4$.
\end{theorem}
\begin{proof}
    Consider the subset
    \[
\mathcal{U}^{d,rs}_\mathrm{conn}(G) \subset \mathcal{M}^{d,rs}_\mathrm{conn}(G)
    \]
  which consists of regularly stable $G$-connections $(E,\nabla)$ such that the underlying $G$-bundle $E$ is regularly stable. Then we have the following map
  \begin{align}\label{forgetful}
  \begin{split}
       p: \mathcal{U}^{d,rs}_\mathrm{conn}(G) &\longrightarrow \mathcal{M}^{d,rs}(G)\\
       (E,\nabla) &\longmapsto E
      \end{split} 
  \end{align}
which forgets the connection structure on the $G$-bundle. Following \cite[Lemma IV.$4$]{F93}, the map $p$ in (\ref{forgetful}) is isomorphic to the torsor of connections on the line bundle
\[
\mathcal{T} \longrightarrow \mathcal{M}^{d,rs}(G)
\]
whose fibers over $E$ are given by $\det H^1(X, \mathrm{ad}(E))$. Since the tangent space of $\mathcal{M}^{d,rs}(G)$ at $E$ is isomorphic to $H^1(X, \mathrm{ad}(E))$ and by Serre duality
\[
H^0(X,\mathrm{ad}(E) \otimes K_X) \cong H^1(X,\mathrm{ad}(E))^\vee,
\]
the cotangent space of $\mathcal{M}^{d,rs}(G)$ at $E$ is isomorphic to $H^0(X,\mathrm{ad}(E) \otimes K_X)$. Since the cotangent space $T^*_E(\mathcal{M}^{d,rs}(G))$ acts on $p^{-1}(E)$ freely and transitively, we conclude that the map $p$ in (\ref{forgetful}) is a torsor for the cotangent bundle of $\mathcal{M}^{d,rs}(G)$. Therefore, the fibers $p^{-1}(E)$ are connected and contractible. Take into account the subsequent exact sequence of homotopy groups
\[
 \pi_k(p^{-1}(E)) \rightarrow  \pi_k(\mathcal{U}^{d,rs}_\mathrm{conn}(G)) \rightarrow \pi_k(\mathcal{M}^{d,rs}(G)) \rightarrow   \pi_{k-1}(p^{-1}(E)),
\]
it directly implies that
\begin{equation}\label{combine1}
     \pi_k(\mathcal{U}^{d,rs}_\mathrm{conn}(G)) \cong \pi_k(\mathcal{M}^{d,rs}(G))
\end{equation}
for all $k\geq 0$. By a similar argument as in \cite[Theorem II.$6$]{F93} and \cite[Theorem $2.5$]{BH12a}, we get that the codimension
\[
\mathrm{codim}\big(\mathcal{M}^{d,rs}_\mathrm{conn}(G) \setminus \mathcal{U}^{d,rs}_\mathrm{conn}(G)\big) \geq g-1.
\]
Since $\mathcal{M}^{d,rs}(G)$ is smooth, we obtain that
\begin{equation}\label{combine2}
    \pi_k(\mathcal{M}^{d,rs}_\mathrm{conn}(G)) \cong \pi_k(\mathcal{U}^{d,rs}_\mathrm{conn}(G))
\end{equation}
for $k=0,\dots, 2g-4$. Thus, combining (\ref{combine1}) and (\ref{combine2}), we get
\[
\pi_k(\mathcal{M}^{d,rs}_\mathrm{conn}(G)) \cong \pi_k(\mathcal{M}^{d,rs}(G))
\]
for $k=0,\dots, 2g-4$.
\end{proof}
\begin{remark}
    The fundamental group of the moduli space $\mathcal{M}^{d,rs}(G)$ of principal $G$-bundles was described in \cite{BMP21}. Therefore, using the above isomorphism we can describe the fundamental group of the moduli space of $G$-connections.
\end{remark}

\section{$\mathrm{SL}_n$ case}\label{lastsection}

In this section, we restrict our attention to the groups $G=\mathrm{GL}(n,\mathbb{C})$ and $\mathrm{SL}(n,\mathbb{C})$. For $G=\mathrm{GL}(n,\mathbb{C})$, the homotopy groups were computed in \cite{BGG08}. So, we consider the case $G=\mathrm{SL}(n,\mathbb{C})$. Let $\mathcal{M}_{\mathrm{conn}}(n,\xi)$ denote the moduli space of semistable connections of rank $n$ and determinant $\xi$ over $X$. If $(n,\deg(\xi))=1$, then $\mathcal{M}_{\mathrm{conn}}(n,\xi)$ is smooth and can be naturally identified with the stable locus.

\begin{theorem}\label{homotopy} Let $X$ be a smooth projective curve of $g\geq 3$ and let $\xi$ be a line bundle of degree $d$ such that $(n,d)=1$. Then
    \begin{enumerate} 
        \item $\pi_1(\mathcal{M}_{\mathrm{conn}}(n,\xi)) = 0;$
        \item $\pi_2(\mathcal{M}_{\mathrm{conn}}(n,\xi)) \cong \mathbb{Z};$
        \item $\pi_k(\mathcal{M}_{\mathrm{conn}}(n,\xi)) \cong \pi_{k-1}(\mathfrak{G})$ for all $2< k \leq 2(n-1)(g-1)-2$, where $\mathfrak{G}$ is the unitary gauge group.
    \end{enumerate}
\end{theorem}
\begin{proof}
    Consider 
    \[
    \mathcal{U}_{\mathrm{conn}}(n,\xi) \subset \mathcal{M}_{\mathrm{conn}}(n,\xi),
    \]
    which consists of all connections $(E,\nabla)$ such that the underlying vector bundle $E$ is semistable. By remark \ref{remark}, we know that the regularly stable locus is the same as the stable locus for $G= \mathrm{SL}(n,\mathbb{C})$, and since $(n,\deg(\xi))=1$, the stable locus coincides with the semistable locus. Now, from the proof of Theorem \ref{maintheorem}, we get that 
    \[
    \pi_k(\mathcal{U}_{\mathrm{conn}}(n,\xi)) \cong \pi_k(\mathcal{M}(n,\xi))
    \]
    for all $k \geq 0$ (see (\ref{combine1})), where $\mathcal{M}(n,\xi)$ denote the moduli space of semistable bundles of rank $n$ and determinant $\xi$ over $X$. Following similar arguments as in \cite[Lemma $3.1$]{BM07} and \cite[Lemma $3.1$]{BM09}, we get that the codimension
    \[
    \mathrm{codim}(\mathcal{M}_{\mathrm{conn}}(n,\xi)\setminus \mathcal{U}_{\mathrm{conn}}(n,\xi)) \geq (n-1)(g-1).
    \]
    Therefore,
    \[
    \pi_k(\mathcal{M}_{\mathrm{conn}}(n,\xi)) \cong \pi_k(\mathcal{U}_{\mathrm{conn}}(n,\xi)) \cong \pi_k(\mathcal{M}(n,\xi))
    \]
    for all $k=0,\dots, 2(n-1)(g-1)-2$. Then the result follows from \cite[Theorem $3.2$]{DU95} (here the assumption is not made that $n$ and $d$ are coprime).
\end{proof}

Let $\mathrm{Tor}_j(\mathcal{M})$ (resp. $\mathrm{Tor}_j(\mathcal{M}_{\mathrm{conn}})$) denote the torsion part of the cohomology group $H^j(\mathcal{M}(n, \xi)$ (resp. $H^j(\mathcal{M}_{\mathrm{conn}}(n, \xi)$).
\begin{theorem}\label{integer}
    The mixed Hodge structures on the cohomology groups (modulo torsion) $H^j(\mathcal{M}_{\mathrm{conn}}(n, \xi),\mathbb{Z})/\mathrm{Tor}_j(\mathcal{M}_{\mathrm{conn}})$ and $H^j(\mathcal{M}(n, \xi),\mathbb{Z})/\mathrm{Tor}_j(\mathcal{M})$ are isomorphic and pure, for all $j \leq 2(n-1)(g-1) - 1$.
\end{theorem}
\begin{proof}
    The morphism
    \begin{equation}\label{isom}
        \pi^* : H^j(\mathcal{M}(n, \xi), \mathbb{Z}) \longrightarrow H^j(\mathcal{U}_{\mathrm{conn}}(n, \xi), \mathbb{Z})
    \end{equation}
    induced by the forgetful map $\pi: \mathcal{U}_{\mathrm{conn}}(n, \xi) \longrightarrow \mathcal{M}(n, \xi)$ is an isomorphism for all $j\geq 0$, as $\pi$ is a torsor with contractible fibers. Consequently, this establishes an isomorphism between two mixed Hodge structures: $H^j(\mathcal{M}(n, \xi), \mathbb{Z})/\mathrm{Tor}_j(\mathcal{M})$ and $H^j(\mathcal{U}_{\mathrm{conn}}(n, \xi), \mathbb{Z})/\mathrm{Tor}_j(\mathcal{U}_{\mathrm{conn}})$, where $\mathrm{Tor}_j(\mathcal{U}_{\mathrm{conn}})$ represents the torsion part of $H^j(\mathcal{U}_{\mathrm{conn}}(n, \xi), \mathbb{Z})$. The relative cohomologies of the pair $(\mathcal{M}_{\mathrm{conn}}(n, \xi), \mathcal{U}_{\mathrm{conn}}(n, \xi))$ induces a long exact sequence
    \begin{alignat}{2}\label{les}
    \begin{split}
    \cdots &\rightarrow H^j(\mathcal{M}_{\mathrm{conn}}(n, \xi), \mathcal{U}_{\mathrm{conn}}(n, \xi), \mathbb{Z}) \rightarrow H^j(\mathcal{M}_{\mathrm{conn}}(n, \xi), \mathbb{Z}) \xrightarrow{\varphi} H^j(\mathcal{U}_{\mathrm{conn}}(n, \xi), \mathbb{Z}) \rightarrow \\ &\rightarrow H^{j+1}(\mathcal{M}_{\mathrm{conn}}(n, \xi), \mathcal{U}_{\mathrm{conn}}(n, \xi), \mathbb{Z}) \rightarrow \cdots, 
    \end{split}
    \end{alignat}
    where $\varphi$ is induced from the inclusion map and it is a morphism of mixed Hodge structures by \cite[p. $43$]{D75}. Since
    \[
    \mathrm{codim}(\mathcal{M}_{\mathrm{conn}}(n,\xi)\setminus \mathcal{U}_{\mathrm{conn}}(n,\xi)) \geq (n-1)(g-1),
    \]
    we get 
    \[
    H^j(\mathcal{M}_{\mathrm{conn}}(n,\xi), \mathcal{U}_{\mathrm{conn}}(n,\xi), \mathbb{Z})=0
    \]
    for all $0\leq j \leq 2(n-1)(g-1)$. Thus $\varphi$ in (\ref{les}) is an isomorphism and hence the composition map with $\pi^*$ as in (\ref{isom})
    \[
    \varphi^{-1} \circ \pi^* : H^j(\mathcal{M}(n, \xi), \mathbb{Z}) \longrightarrow H^j(\mathcal{M}_{\mathrm{conn}}(n, \xi), \mathbb{Z})
    \]
    is an isomorphism for all $j \leq 2(n-1)(g-1) -1$. Therefore, we obtain an isomorphism
    \[
    H^j(\mathcal{M}_{\mathrm{conn}}(n, \xi),\mathbb{Z})/\mathrm{Tor}_j(\mathcal{M}_{\mathrm{conn}}) \cong H^j(\mathcal{M}(n, \xi),\mathbb{Z})/\mathrm{Tor}_j(\mathcal{M})
    \]
    which preserves the mixed Hodge structures, for all $j \leq 2(n-1)(g-1) -1$.
\end{proof}

\begin{theorem}{($G=\mathrm{SL}(n,\mathbb{C})$)}
    The (rational) cohomology groups of the fibers of the map $\pi: \mathcal{M}^{d, rs}_\mathrm{Hod}(G) \longrightarrow \mathbb{C}$ (as in (\ref{rsproj})) are isomorphic, preserving the associated mixed Hodge structures. The mixed Hodge structures in these cohomology groups are pure. 

    In particular, the mixed Hodge structures on $H^\bullet(\mathcal{M}_\mathrm{Higgs}^{d, rs}(G), \mathbb{Q})$ and $H^\bullet(\mathcal{M}_\mathrm{conn}^{d, rs}(G), \mathbb{Q})$ are isomorphic and pure.

    These statements hold for the (rational) cohomology groups with compact support.
\end{theorem}
\begin{proof}
    Since, $\mathcal{M}^{d, rs}_\mathrm{Hod}(G)$ is smooth and $\pi: \mathcal{M}^{d, rs}_\mathrm{Hod}(G) \longrightarrow \mathbb{C}$ is a smooth map (i.e. a surjective submersion), from \cite[Theorem $7.2.1$]{HLR11} we get
    \[
    H^\bullet(\pi^{-1}(t_1), \mathbb{Q}) \cong H^\bullet(\pi^{-1}(t_2), \mathbb{Q})
    \]
    which preserves the mixed Hodge structures, for all $t_1,t_2 \in \mathbb{C}$ and the Hodge structures are pure. For $t_1=0$ and $t_2=1$, we obtain
    \[
    H^\bullet(\mathcal{M}_\mathrm{Higgs}^{d, rs}(G), \mathbb{Q}) \cong H^\bullet(\mathcal{M}_\mathrm{conn}^{d, rs}(G), \mathbb{Q})
    \]
    which preserves the mixed Hodge structures and the Hodge structures are pure. The statement concerning compactly supported cohomology groups is guaranteed by the application of the Poincar\'e duality (as stated in \ref{poincare}).
\end{proof}

\section*{Acknowledgment}
 The author is supported by the INSPIRE faculty fellowship (Ref No.: IFA22-MA 186) funded by the DST, Govt. of India.

\end{document}